\documentclass[9pt,epsfig]{amsart}
\usepackage{amssymb, adjustbox, enumerate, amsbsy, stmaryrd}
\usepackage[dvipsnames]{xcolor}
\usepackage{amsmath,wasysym}
\usepackage{comment}
\usepackage{caption}
\usepackage[mathscr]{eucal}
\usepackage{amsthm}
\usepackage{graphicx}
\usepackage {amscd}
\usepackage {epic}
\usepackage[all,color]{xy}
\usepackage[alphabetic]{amsrefs}
\usepackage{color}

\usepackage{graphpap, color}

\usepackage[mathscr]{eucal}
\usepackage{mathrsfs}
\usepackage{graphicx}

\usepackage{tikz}

\newtheorem{prop}{Proposition}[section]

\newtheorem{assu}[prop]{Assumption}

\newtheorem{coro}[prop]{Corollary}
\newtheorem{defi}[prop]{Definition}

\newtheorem{exam}[prop]{Example}

\newtheorem{lemm}[prop]{Lemma}
\newtheorem{pf-thm}[prop]{proof of theorem}
\newtheorem{rema}[prop]{Remark}

\newtheorem{theo}[prop]{Theorem}
\newtheorem{ques}[prop]{Question}
\newtheorem*{ack}{Acknowledgments}

\def\cX{\mathcal{X}}

\def\sO{{\mathscr O}}

\def\sL{{\mathscr L}}

\def\sO{\mathscr{O}}

\newcommand{\CC}{\mathbb{C}}

\newcommand{\NN}{\mathbb{N}}
\newcommand{\PP}{\mathbb{P}}
\newcommand{\QQ}{\mathbb{Q}}
\newcommand{\RR}{\mathbb{R}}
\newcommand{\ZZ}{\mathbb{Z}}

\newcommand{\bi}{\mathbf{i}}

\newcommand{\fc}{\mathfrak{c}}

\newcommand{\fg}{\mathfrak{g}}

\newcommand{\fk}{\mathfrak{k}}

\newcommand{\ft}{\mathfrak{t}}

\def\aut{\mathrm{Aut}}

\def\chow{\mathrm{Chow}}

\def\conv{\mathrm{Conv}}

\def\gr{\mathrm{graph}}

\def\dd{\mathrm{partial}}

\def\SL{\mathrm{SL}}

\def\sym{\mathrm{Sym}}

\def\vol{\mathrm{vol}}

\def\Ad{\mathrm{Ad}}

\def\la{\langle}
\def\ra{\rangle}

\def\dd{\partial}

\newcommand{\red}{\color{red}}

\newcommand{\lam}{\lambda}

\newcommand{\al}{\alpha}

\newcommand{\de}{\delta}
\newcommand{\De}{\Delta}

\newcommand{\si}{\sigma}

\newcommand{\tri}{\triangle}

\input{xy}
\xyoption{all}

\begin{document}

\title{Asymptotic Chow stability of toric Del Pezzo surfaces}

\author{King-Leung Lee}
\address{Department of Mathematics and Computer Science\\
           Rutgers University, Newark NJ 07102-1222\\ USA}
\email{kl465@rutgers.edu}

\author{Zhiyuan Li}
\address{Shanghai Center for Mathematical Science\\ Fudan University, 220 Handan
Road, Shanghai, 200433 \\China}
\email{zhiyuan.li@fudan.edu.cn}

\author{Jacob Sturm}
\address{Department of Mathematics and Computer Science\\
           Rutgers University, Newark NJ 07102-1222\\ USA}
\email{sturm@rutgers.edu}

\author{Xiaowei Wang}
\address{Department of Mathematics and Computer Science\\
           Rutgers University, Newark NJ 07102-1222\\ USA}
\email{xiaowwan@rutgers.edu}

\date{\today}
\maketitle
\begin{abstract}
In this short note, we study the asymptotic Chow  polystability of toric Del Pezzo surfaces appear in the moduli space of K\"ahler-Einstein Fano varieties constructed in \cite{OSS}.
\end{abstract}

\section{Introduction}
Since the invention of geometric invariant theory
\cite{MFK} by David Mumford, GIT has been successfully applied to the construction of various kinds of moduli
spaces, e.g. moduli spaces of stable vector bundles over a projective curves and of moduli spaces of polarized varieties $(X,L)$.
In particular, when $X$ is
a canonically polarized manifold, it was shown by Mumford and Gieseker in dimension
1, by Gieseker \cite{Gie77} in dimension 2, and in arbitrary dimensions by Donaldson \cite{Don01}  (making use of the work of  Aubin , Yau \cite{Aub76, Yau} and Zhang \cite{ Zha 96})
that $(X,L=\sO_X(K_X))$  is asymptotically Chow stable (see also \cite{PS2004}). That
is, given a smooth canonically polarized variety $(X,\sO_X(K_X))$,  that there exists
an $r_0$ such that $(X,\sO_X(rK_X))$
is Chow stable for any $r\geq r_0$. More generally, if $(X,L)$ is a polarized manifold, GIT also plays a role in the existence of constant scalar curvature K\"ahler metrics in the class of $L$ (see for example the survey article \cite{PS2009}). 
\vskip .1in

In order to compactify the moduli space it is necessary to include {\em singular} varieties (e.g. by stable reduction theorem for curves). In general, it is quite difficult to extend above works to singular varieties, even for the $\dim=1$ case (cf.\cite{LiW15, Gie77}).
On the other hand, it was shown in \cite{WX14}, that {\em asymptotic} Chow stability does not form a proper moduli space in general by exhibiting some explicit punctured families of {\em canonical polarized} varieties without  asymptotic  Chow semi-stable filling.  However, in \cite{LWX2014}, a proper moduli space of smoothable K-semistable {\em Fano varieties} is constructed. It is a natural to ask  whether or not the moduli space of $\QQ$-Fano varieties can be realized as asymptotic GIT moduli space at least when the  dimension is {\em small}. 
\footnote{We remark that Ono, Sano and Yotsutani succeeded in constructing a $\dim=7$ toric Fano K\"ahler-Einstein manifold that is not asymptotic Chow stable in \cite{OSN12}. But that did not rule out the asymptotic GIT completely, see Remark \ref{osn} for more explanation.}  To answer this question, one needs  to  understand first when $\dim=2$, in particular those Fano varieties appear in the moduli spaces of K-semistable Del Pezzo surfaces constructed \cite{OSS}. For smooth K\"ahler-Einstein Fano manifolds, by  Mabuchi's extension \cite{Mab2004} of Donaldson's work \cite{Don01} we know that they are all asymptotic Chow polystable provided their automorphism groups are semi-simple. 
Unfortunately,  it seems quite difficult to extend Donaldson and Mabuchi's approach in \cite{Don01, Mab2004} to singular Fano varieties, at least to the best of our knowledge so far,  there is {\em not a single non-smooth} example of $\QQ$-Fano varieties whose asymptotic Chow stability is {\em known}. In this note we want to close this gap by studying the asymptotic Chow stability of some singular toric Del Pezzo surfaces. {The original motivation was the following question which was asked of us by Odaka and Laza}.

\begin{ques}\label{Q}
Is the K-polystable cubic surface $X:=\{xyz=w^3\}\subset \PP^3$  asymptotically Chow stable?
\end{ques}

To state our main result, let
\begin{enumerate}\label{X}
\item $(X_1=\PP^2/(\ZZ/9\ZZ),\sO_{X_1}(1):=\sO_{X_1}(-3K_{X_1}))\subset (\PP^6,\sO_{\PP^6}(1))$ with the $\ZZ/9\ZZ=\la\xi=\exp2\pi\sqrt{-1}/9\ra$-action generated by  $\xi\cdot [z_0, z_1,z_2]=[z_0,\xi,z_1,\xi^{-1},z_2]$.
\item $(X_2=\PP^1\times\PP^1/(\ZZ/4\ZZ),\sO_{X_2}(1):=\sO_{X_2}(-2K_{X_2}))\subset (\PP^6,\sO_{\PP^6}(1))$ with the $\ZZ/4\ZZ=\la\xi\ra$-action generated by  $\xi\cdot ([z_1,z_2],[w_1,w_2])=([\sqrt{-1} z_1,z_2],[-\sqrt{-1} w_1,w_2])$.
\item $(X_3=\{xyz=w^3\},\sO_{X_3}(1):=\sO_{X_1}(-K_{X_1}))\subset(\PP^3,\sO_{\PP^3}(1))$;
\item $(X_4=Q_1\cap Q_2,\sO_{X_4}(1):=\sO_{X_2}(-K_{X_2}))\subset (\PP^4,\sO_{\PP^4}(1))$ with
$$\displaystyle
 \left\{ \begin{array}{llll}
         Q_1: &z_0z_1+z_2z_3+z_4^2 &=0& \\
        Q_2: & \lam z_0z_1+\mu z_2z_3+z_4^2&=0 & \lam\ne \mu
        \end{array} \right.
,$$

\end{enumerate}
These are the {\em only}  $\QQ$-Gorenstein smoothable toric  K\"ahler-Einstein (i.e. K-polystable)  Del Pezzo surfaces of $\deg=1,2,3,4$ thanks to the work of \cite[Theorem 2.3.3]{Spotti12}. In particular, they are  parametrized in the proper moduli spaces  constructed in \cite[Theorem 4.1, 4.2, 4.3, 5.13, 5.28]{OSS}. Then our main result is the following:
\begin{theo}\label{main}
Let  $(X_i,\sO_{X_i}(k))$ be one in the list \eqref{X}. Then $(X_i,\sO_{X_i}(k))$ is
\begin{enumerate}
\item Chow unstable  for any $k\geq 1$ when $i=1$;
\item Chow polystable for $k\geq 2$ when $i=2$;
\item Chow polystable for $k\geq 1$ when $ i=3,4$.
\end{enumerate}
\end{theo}

Our paper is organized as follows: in section two we review some basic facts of GIT, in particular, we reduce the checking of stability to a purely combinatorial problem thanks to the fact that the $X_i$ are toric. In section three, we will carry out the main estimate that is needed for the proof of the last case of  Theorem \ref{main}. In section four we  extend  the main estimate used in section three and prove the second cases of Theorem \ref{main}. It turns out this is the most delicate calculation. In the last section, we establish the first case by showing the non-vanishing of  Chow weight of the torus action. We want to remark that examples of  asymptotic Chow unstable Fano toric K\"ahler-Einstein manifolds were first found in \cite{OSY}.

\begin{ack}The second and last author want to express their gratitute to AIM, San Jose for the excellent research environment they provided. The project is originally started during both the second and last authors were participating the AIM workshop on {\em Stability and moduli spaces} in January, 2017. In particular, we want to thank Radu Laza for proposing  Question \ref{Q} explicitly. We also thank Cristiano Spotti for informing us Theorem 2.3.3 in his thesis.  The last author wants to thank IHES for the fantastic research environment, where part of the writing was carried out. The work of the last author was partially supported by  a Collaboration Grants for Mathematicians from Simons Foundation:281299 and NSF:DMS-1609335.
\end{ack}

\section{Basics on GIT and symplectic quotient}
In this section we include a symplectic quotient proof of Kempf's instability result \cite[Corollary 4.5]{Kemp1978}, which reduces checking of Chow stability of a projective variety to a smaller group provided the variety admits a large symmetry group.

\subsection{Kempf's instability theory}
Let $G$ be a reductive algebraic group acting on a polarized pair  $(Z,\sO_{Z}(1))$, i.e. $\sO_Z(1)$ is $G$-linearzed. Let $K<G$ be a {\em maximal compact subgroup}. Fixing a $K$-invariant Hermitian metric with a positive curvature form $\omega$ on $L$, we obtain a holomorphic Hamiltonian $K$-action on $(Z,\omega)$ with moment map
$$\mu_K:Z\longrightarrow \fk.$$
Let $z\in Z$ be a point with stabilizer $G_z<G$.
\begin{defi}
We say a $G$-orbit $G\cdot z\in Z$ is  {\em $G$-extremal} with respect to the $G$-action on $(Z,\sO_{Z}(1))$ if and only if there is a maximal compact subgroup $K<G$ and a $h\in G$ as above such that $\mu_K(h\cdot z)\in \fk_{h\cdot z}$, the stabilizer of $h\cdot z$ in {$\fk$}. 
\footnote{Notice that, if one translate the K\"ahler form $\omega$ on $Z$ by a $h\in G$ then  the above definition can be reformulated as following: for any {\em prefixed} maximal compact $K<G$ there exists a $h\in G$ such that $\mu_K(h\cdot z)\in K_{h\cdot z}$. }
This is equivalent to saying that $h\cdot z$ is a critical point of
$$|\mu_K|_\fk^2=\la\mu_K,\mu_K\ra_\fk :Z\longrightarrow \RR$$
 where $\la\cdot,\cdot\ra_\fk$ is a $K$-invariant inner product on $\fk$. 
 We say $z$ is {\em $G$-polystable} if there is a  maximal compact subgroup $K<G$ such that $\mu_\fk(z)=0$.
\end{defi}

Now we are ready to give a simple and symplectic quotient proof of a slight improvement of Kempf's instability Theorem  \cite[Corollary 4.5]{Kemp1978}.
\begin{theo}\label{kempf}
Let $G_0<G_z$ be a {\em reductive} subgroup. Then $G\cdot z$ is an {\em $G$-extremal  (resp. poly-stable)}  if and only if  $C(G_0)\cdot z$ is  {\em $C(G_0)$-extremal (resp.  poly-stable)}  with respect to the $C(G_0)$-action, induced by the embedding $i:C(G_0)\hookrightarrow G$ on  $(Z,\sO_{Z}(1))$, where $C(G_0)<G$ is the {\em centralizer} of the $G_0$ in $G$.
\end{theo}
\begin{proof}
%

Let us  fix a maximal compact subgroup  $K<G$ such that $(K_{0})^{\CC}=G_0$ with $K_{0}:=K\cap G_0$. We define
$$K_H:=C(K_{0})=\{g\in K \mid \mathrm{Ad}_gh=h,\forall h\in K_{0}\}<K,$$
the {\em centralizer} of $K_{0}$ in $K$ and  $H:=K_H^\CC$.
Suppose $H\cdot z$ is  $H$-extremal then there is a $h\in H$
$$
\mu_{K_H}(h\cdot z)=i^\ast \mu_{K}(h\cdot z)\in \fk_H\cap \fk_{h\cdot z}, (\text{\em resp.}=0 \text{ if $z$ is $H$-polystable})
$$
where $ i^\ast:\fk\rightarrow \fk_H$ be the {\em orthogonal projection} with respect to a $\Ad_K$-invariant inner product $\la\cdot,\cdot\ra_\fg$ on $\fg$.

Since $h\in H=C(G_0)$, we have $\Ad_h G_0=G_0<G_{h\cdot z}$.  
Without loss of generality we may assume that $h=e$, the identity (i.e. replace $h\cdot z$ by $z$ from  the beginning).
Then \begin{equation}\label{perp}
i^\ast\mu_K(z)\in \fk_H\cap \fk_z, (\text{\em resp. }\mu_K(z)\perp \fk_H \text{ if $z$ is $H$-polystable}).
\end{equation}
On the other hand, for any $k\in K_{0}<G_0<G_z$ we have
$$
\mu_K(z)=\mu_K(k\cdot z)=\Ad_k\mu_K(z),
$$
from which we deduce that  $\mu_K(z)\in \fc(K_{0})=\fk_H$. This combined with \eqref{perp} implies that
$$
\mu_K(z)\in \fk_z,(\text{\em resp.}=0 \text{ if $z$ is $H$-polystable})
$$ i.e. $z$ is $G$-extremal ({\em resp.} $G$-polystable).

Conversely, suppose $G\cdot z$ is extremal. Then   we have 
$\|\mu_K(z)\|=\min_{G\cdot z}\|\mu_K\|$
by \cite[Theorem 6.2]{Ness84} and  $\mu_K(z)\in {\ \fc(K_z)}\subset\fk_H$ (, where $\fc(K_z)$ is the Lie algebra of the centralizer of $C(K_z)<K$) by \cite[Theorem 10]{Wang04}, from which we conclude
$$\|\mu_K(z)\|=\min_{G\cdot z}\|\mu_K\|=\min_{H\cdot z}\|\mu_{K_H}\|=\|\mu_{K_H}(z)\|.$$
Thus $C(G_0)\cdot z$ is extremal and our proof is completed.
\end{proof}

\begin{coro}Let us continue with the notation in the Theorem \ref{kempf}. Then $z\in Z$ is $G$-semistable if and only if $z$ is $C(G_0)$-semistable.
\end{coro}
\begin{proof}
By our assumption $z$ is $H$-semistable with $H=C(G_0)$, so there is a
$$z_0\in \overline{H\cdot z}\subset \overline{G\cdot z}\subset Z$$
such that $z_0$ is $H$-polystable. By Theorem \ref{kempf}, we know $z_0$ is $G$-polystable and our proof is completed.
\end{proof}
\bigskip

\subsection{Toric varieties}
Let $\triangle\subset \RR^n$ be any {\em convex polytope} and we will introduce {\em cone}  $\mathrm{PL}(\tri;k)$  in $C^0(k\tri,\RR)$, the space of continuous functions on $k\tri$. To begin with, let $\phi:k\tri\cap\ZZ^n\to \RR$ be any function and   define:
$$\gr_\phi:=\conv \left\{ \bigcup_{x\in k\tri\cap\ZZ^n} \{\left.(x,t)\in \RR^n\times \RR \right| t\leq \phi(x)\}\right\} $$
the {\em convex hull} of the set $\bigcup_{x\in k\tri\cap\ZZ^n} \{\left.(x,t)\in \RR^n\times \RR \right| t\leq \phi(x)\}$.
\begin{defi} Let  $\triangle\subset \RR^n$ be any convex polytope. We define 
\begin{enumerate}
\item A function $C^0(k\tri,\RR)\ni f_\phi :k\tri\to \RR$ is said to be {\em associated} to a  $\phi:k\tri\cap\ZZ^n\to \RR$ if 
$$
f_\phi(x):=\max\{t\mid (x,t)\in \gr_\phi\}:k\tri\longrightarrow \RR.
$$
\item We define the {\em cone} 
\begin{equation}\label{pl}
\mathrm{PL}(\tri;k):=\{f_\phi \ \left|\ \phi:k\tri\cap\ZZ^n\to \RR\right. \}\subset C^0(k\tri,\RR).
\end{equation}
\end{enumerate}
\end{defi}

Now to apply Theorem \ref{kempf} to our situation, let $(
X_{\triangle },L_{\triangle }) $ be any polarized toric variety ({\em not necessarily smooth}) with moment polytope $\triangle$. 
Let $\aut(X_\tri)$ denote the {\em automorphism} of the pair $(X_\tri,L_\tri)$, then $T=(\CC^\times)^{n}<\aut(X_\tri)$ is a maximal torus.  
\begin{defi}
Let $(X_\tri,L_\tri)$ be a polarized toric variety with moment polytope $\tri$, we define 
the  {\em Weyl group} $W_\tri:=N(T)/T$ with
$$T=(\CC^\times)^{n}<N(T):=\{g\in \aut(X_\tri, L_\tri)\mid g\cdot T\cdot g^{-1}=T\}<\aut(X_\tri)$$
 being  the  {\em normalizer } of $T<\aut(X_\tri)$.
Clearly, $W_\tri$ acts on $\tri\subset \RR^n\cong \ft$ via the adjoint action. 
\end{defi}
Consider a projective embedding 
$$
(X_\tri, L_\tri^k)\longrightarrow (\PP^{N},\sO_{\PP^N}(1)) 
$$
with  
$$N+1=\chi_\tri(k)=\dim H^0(X_\tri, L^k_\tri)=|k\tri\cap\ZZ^n| \text{ and }\deg X_\tri=d.$$
Let 
$$
\chow_k(X_\tri)\!\!:=\!\!\left\{\left.(H_0,\cdots, H_{n})\in ((\PP^N)^\vee)^{n+1}\right| H_0\cap\cdots\cap H_n\cap X\ne\varnothing\right\}
\in \PP^{d,n;N}\!\!\!:=\PP(\sym^d(\CC^{N+1})^{\otimes (n+1)})
$$
denote the  {\em $k$-th Chow form} associated to the embedding above.  
With those notation understood, we state a result due to H. Ono  \cite{Ono2013}.
\begin{theo}[Theorem 1.1, \cite{Ono2013}]\label{ono}
Let $(X_{\triangle },L_{\triangle })$ be a polarized toric variety (not necessarily smooth) with moment polytope $\triangle \subset \mathbb{R}^{n}$. For a fixed positive integer $k$, $\chow_k(X_\tri)$ of \
$(X_{\triangle },L_{\triangle }^{k})\subset (\PP^N,\sO_{\PP^N}(1)) $ is  {\em polystable} with respect to the action of the {\em subgroup of diagonal matrices}  in $\SL(N+1)$ 
if and only if%
\begin{equation}\label{ch-wt}
\frac{1}{\vol(
k\triangle ) }\int_{k\triangle }g-\frac{1}{\chi_{\triangle
}( k) }\sum_{x\in \triangle \cap \ZZ%
 ^{n} }g( x) \geq 0,
\end{equation}
for any $g\in \mathrm{PL}(\triangle ;k)$ with equality  if and only if $g$ being affine.
\end{theo}
Now let $(Z,\sO_Z (1))=(\PP^{d,n;N},\sO_{\PP^{d,n;N}}(1))$, $G=\SL(N+1)$ and 
$G_0=N(T) <G_{\chow_k(X_\tri)}=\aut(X_\tri)$. Then the centralizer $C(G_0)<\SL(N+1)$ is contained in a {\em maximal torus} (e.g.  the subgroup of diagonal matrices) of $\SL(N+1)$. In particular, 
Theorem \ref{ono} together with Theorem \ref{kempf} then imply the following 
\begin{coro}\label{W-inv}
Let $(X_\tri,L_\tri)$ be a polarized toric variety with moment polytope $\tri$ as above and $W=W_\tri$ be the Weyl group. Then for  any $k\in \NN$, $(
X_{\triangle },L_{\triangle }^{k}) $ is  Chow polystable (i.e. $\chow_k(X_\tri)\in \PP^{d,n;N}$ is GIT polystable with respect to the $\SL(N+1)$-action on $(\PP^{d,n;N},\sO_{\PP^{d,n;N}}(1))$)
if and only if%
\eqref{ch-wt} holds for any
$$
g\in \mathrm{PL}(
\triangle ;k)^W=\{g\in \mathrm{PL}(
\triangle ;k)\mid g(w\cdot x)=g(x)\  \forall w\in W  \},$$
with equality if and only if $g$ being affine.
\end{coro}

 Theorem \ref{ono} was originally proved in \cite{Ono2013} for {\em integral} Delzant polytope by applying the powerful machinery developed by Gelfand-Kapranov-Zelevinsky in \cite{GKZ}. Here for reader's convenience, we give a  slightly simpler and  more direct proof.
\begin{proof}[Proof. of Theorem \ref{ono}]
Without loss of generality, we may assume $L_\tri$ is very ample and $k=1$. Also since the left hand side of \eqref{ch-wt} is invariant under adding a constants, we may assume $g\geq 0$.

Let $(\cX,\sL)\to \PP^1$ be any $T$-equivariant test configuration of $(X_\triangle, L_\triangle)$. So $\cX$ is a $n+1$-dimensional toric variety. Let
$$\triangle_g:=\{ (x, y)\in (\triangle\cap \ZZ^n)\times \RR_{\geq 0} \mid  0\leq y\leq g(x)\}\subset \RR^n\times \RR_{\geq 0}$$
be the moment polytope of  $\cX$, where $g$ is a  non-negative  rational piecewise-linear concave function defined over $\triangle$. Then we have
\begin{equation}\label{tri-g}
\vol(\tri_g)=\int_{\triangle }g(x)dx\text{   and   }\chi_{\tri_g}(1)-\chi_\tri(1)=\sum_{x\in \triangle \cap \ZZ^{n} }g( x).
\end{equation}

By the proof of \cite[Proposition 4.2.1]{Don02}, we know the {\em weight} of the $\CC^\times$-action on $\wedge^{\chi_\triangle(m)} H^0(\cX_0,\sL^m|_{\cX_0})$ is given by
\begin{equation}\label{w-m}
w_m=\chi_{\triangle_g}(m)-\chi_{\triangle}(m)
\end{equation}
with asymptotic expansions (cf. \cite[Propostion 4.1.3 and equation (4.2.2) ]{Don02})
\begin{equation}
\chi_\tri(m)=m^n \vol(\tri)+O(m^{n-1}) \text{  and  }\chi_{\tri_g}(m)=m^{n+1} \vol(\tri_g)+O(m^{n}) .
\end{equation}
On the other hand, the {\em Chow weight} for the degeneration $(\cX,\sL)\to \PP^1$  is given by the {\em normalized leading coefficient (n.l.c)} of the top degree term $\displaystyle\frac{m^{n+1}}{(n+1)!}$ in the degree $n+1$ polynomial of $m$:
$$\displaystyle w_m-m\chi_\triangle(m)\frac{w_1}{\chi_\triangle(1)}, $$
where the second term is added in order to {\em normalize} the $\CC^\times$-action on $H^0(\cX_0,\sL|_{\cX_0})$ to be {\em special linear} (cf.\cite[Theorem 3.9 and equation (3.8)]{RT2007}). Then by \eqref{w-m} we obtain
\begin{eqnarray*}
&&w_m-m\chi_\triangle(m)\frac{w_1}{\chi_\triangle(1)}\\
&=&\chi_{\triangle_g}(m)-\chi_{\triangle}(m)-m\chi_\triangle(m)\frac{\chi_{\tri_g}(1)-\chi_\tri(1)}{\chi_\triangle(1)}\\
&=&m^{n+1}\vol(\tri_g)-m^{n+1}\vol(\tri)\frac{\chi_{\tri_g}(1)-\chi_\tri(1)}{\chi_\triangle(1)}+O(m^n)\\
&=&m^{n+1}\vol(\tri)\left(\frac{1}{\vol(\triangle ) }\int_{\triangle }g-\frac{1}{\chi_{\triangle}( 1) }\sum_{x\in \triangle \cap \ZZ^{n} }g( x)\right)+O(m^n)
\end{eqnarray*}
where for the last identity we have used \eqref{tri-g}. Hence the Chow weight for the $T$-equivariant test configuration 
$(\cX,\sL)\to\PP^1$ is precisely
$$
(n+1)!\vol(\tri)\left(\frac{1}{\vol(\triangle ) }\int_{\triangle }g-\frac{1}{\chi_{\triangle}( 1) }\sum_{x\in \triangle \cap \ZZ^{n} }g( x)\right), 
$$
and our proof is completed.
\end{proof}

\begin{coro}[Corollary 4.7, \cite{Ono2013}]\label{co-ono}
If $(X_\triangle, L_\triangle^k)$ is Chow semistable for $k\in \NN$ then 
\begin{equation}\label{bary}
\frac{1}{\chi_\triangle (k)}\sum_{x \in k\triangle\cap \ZZ^n} x=\frac{1}{\vol(k\triangle)}\int_{k\triangle} x dx.
\end{equation}
\end{coro}
\begin{rema}
The identity \eqref{bary} is equivalent to the vanishing of Chow
weight for the group $T=(\CC^\times)^{n}<\aut(X_\tri)$.
 In particular,   \eqref{bary} implies that the left hand side of \eqref{ch-wt}  is invariant under addition of an {\em affine function} to $g$.
\end{rema}


\begin{exam}
Let $(X_\tri,L_\tri)=( \PP^{1},\sO_{\PP^1}(1) ) $ then
\begin{eqnarray}\label{P1}
&&\frac{1}{\vol( [
0,k] ) }\int_{0}^{k}g-\frac{1}{\chi_{\triangle }(
k) }\sum_{{x\in k\triangle \cap }\ZZ}g( x) \nonumber 
=\int_{0}^{k}g-\frac{1}{k+1}\sum_{i=0}^{k}g(i) \geq 0, \ \forall g \text{ concave }
\end{eqnarray}
follows from the fact that
\begin{equation}\label{1/2}
\frac{1}{k}\left(\frac{1}{2}g(0)+g(1)+\cdots+g(k-1)+\frac{1}{2}g(k)\right)
\geq\frac{1}{k+1}\left(g(0)+g(1)+\cdots+g(k-1)+g(k)\right),\ \ \forall g\geq 0.
\end{equation}
\end{exam}

\section{$X_3$ and $X_4$.}\label{m-est}
In this section, we will treat {$X_{\tri_3}$} and $X_{\tri_4}$ simultaneously since both $\tri_i, i=3,4$ allows a decomposition of $\tri_i$ with the {\em same} fundamental domain $\tri_0$ (cf. Figure \ref{de12}). Let
\begin{enumerate}
\item $(X_{\tri_3},L_{\tri_3})=(X_3,\sO_{X_3}(-K_{X_3}))=\{xyz=w^3\}\subset(\PP^3,\sO_{\PP^3}(1))$.
\item $(X_{\tri_4},L_{\tri_4})=(X_4,\sO_{X_4}(-K_{X_4}))=Q_1\cap Q_2\subset (\PP^4,\sO_{\PP^4}(1))$ with
\begin{equation}\label{d=4}
 \left\{ \begin{array}{llll}
         Q_1: &z_0z_1+z_2z_3+z_4^2 &=0& \\
        Q_2: & \lam z_0z_1+\mu z_2z_3+z_4^2&=0.& \lam\ne \mu
        \end{array} \right.
\end{equation}
\end{enumerate}
with moment polytope $\tri_i,i=3,4$ given in Figure \ref{de12}.
\begin{center}
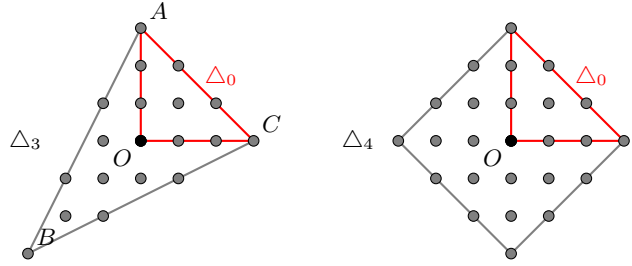
\begin{figure}[h!]
\begin{tikzpicture}[x=.5cm,y=.5cm]
\draw[red, thick] (0,3) -- (0,0);
\draw[red, thick] (0,0) -- (3,0);
\draw[red,thick] (0,3) -- (3,0);
\node[text width=1.5cm] at (3.2,1.7) {\red $\triangle_0$};

\draw[gray, thick] (0,3) -- (-3,-3);
\node[text width=1.5cm] at (-2,0) { $\triangle_3$};

\draw[gray, thick] (-3,-3) -- (3,0);

\foreach \Point/\PointLabel in
{
(-3,-3)/B,
(-2,-1)/, (-2,-2)/,
(-1,1)/, (-1,0)/,(-1,-1)/, (-1,-2)/,
(0,3)/A,(0,2)/, (0,1)/, (0,0)/, (0,-1)/,
(1,2)/, (1,1)/, (1,0)/, (1,-1)/,
(2,1)/, (2,0)/, (3,0)/C}
\draw[fill=gray] \Point circle (2pt) node[above right] {$\PointLabel$};
\foreach \Point/\PointLabel in
{
(0,0)/O, 
 }
\draw[fill=black] \Point circle (2pt) node[below left] {$\PointLabel$};
\end{tikzpicture}
\begin{tikzpicture}[x=.5cm,y=.5cm]
\draw[red, thick] (0,3) -- (0,0);
\draw[red, thick] (0,0) -- (3,0);
\draw[red,thick] (0,3) -- (3,0);
\node[text width=1.5cm] at (3.2,1.7) {\red $\triangle_0$};

\draw[gray, thick] (0,3) -- (-3,0);
\draw[gray, thick] (-3,0) -- (0,-3);
\draw[gray, thick] (3,0) -- (0,-3);
\node[text width=1.5cm] at (-3,0) { $\triangle_4$};

\foreach \Point/\PointLabel in
{
(0,-3)/,
(-1,-2)/,(0,-2)/, (1,-2)/,
(-2,-1)/, (-1,-1)/,(0,-1)/, (1,-1)/,(2,-1)/,
(-3,0)/,(-2,0)/, (-1,0)/, (0,0)/, (1,0)/,(2,0)/,(3,0)/,
(-2,1)/, (-1,1)/, (0,1)/, (1,1)/,(2,1)/,
(-1,2)/, (0,2)/, (1,2)/, (0,3)/}
\draw[fill=gray] \Point circle (2pt) node[above right] {$\PointLabel$};
\foreach \Point/\PointLabel in
{
(0,0)/O, 
 }
\draw[fill=black] \Point circle (2pt) node[below left] {$\PointLabel$};
\end{tikzpicture}
\caption{$\triangle_0\subset\triangle_1$ and $\tri_0\subset\tri_2$.}\label{de12}
\end{figure}
\end{center}

Notice both $\tri_i,i=3,4$  are invariant under the action of Weyl group $W_i:=W_{\tri_i},i=3,4$ respectively, where
$$W_3=D_3=\left\la\si_3:=\begin{bmatrix}0 & -1\\ 1 & -1\end{bmatrix},\begin{bmatrix}0 & 1\\ 1 & 0\end{bmatrix} \right \ra
\text{ and }W_4=D_4=\left\la\si_4:=\begin{bmatrix}0 & -1\\ 1 & 0\end{bmatrix},\begin{bmatrix}0 & 1\\ 1 & 0\end{bmatrix} \right \ra <\mathrm{GL}(2,\ZZ).$$
To prove Theorem \ref{main}, first we  establish the necessary condition \eqref{bary}, which is a consequence of the following

\begin{lemm}\label{ma-0}
Let $\mu$ be any measure defined on $\tri$  and  $\si\in\mathrm{SL}(2,\RR)$ be a element of {\em  order } $d$ satisfying
\begin{enumerate}
\item $\si(\tri)=\tri$;
\item $\si^\ast d\mu=d\mu$.
\end{enumerate}
 Suppose further $\tri$ admits a decomposition
$\displaystyle\tri=\bigsqcup_{i=0}^{d-1} \si^k(\tri_0)$ such that $\si^i(\tri_0^\circ)\cap\si^j(\tri_0^\circ)=\emptyset$ for $i\ne j$, where $\tri_0^\circ$ denotes the interior of a closed subset $\tri_0\subset\tri$.
Then
$$
\int_\tri x d\mu(x)=0.
$$
\end{lemm}
\begin{proof}
By our assumption that $\si\in\mathrm{SL}(2,\RR)$ of order $d+1$, we have
$$
\sum_{i=0}^{d-1}\si^k=0\in \mathrm{SL}(2,\RR).
$$
Hence
\begin{eqnarray*}
\int_{\tri}x d\mu(x)
&=&\sum_{i=0}^d\int_{\si^k(\tri_0)}x d\mu(x)=\sum_{i=0}^d\int_{\tri_0}(x\circ\si^k)\cdot(\si^k)^\ast d\mu(x)\\
&=&\sum_{i=0}^d\int_{\tri_0}(x\circ \si^k)\cdot d\mu(x)=\int_{\tri_0}x\circ(\sum_{i=1}^d\si^k)d\mu(x)=0
\end{eqnarray*}
and our proof is completed.
\end{proof}

\tikzset{
  font={\fontsize{9pt}{12}\selectfont}}
\begin{center}
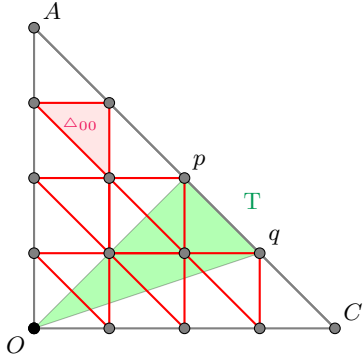
\begin{figure}[h!]
\begin{tikzpicture}[x=1cm,y=1cm]
\draw[fill=red, opacity=0.1]  (0,3) -- (1,3) -- (1,2) -- cycle;
\draw[fill=green, opacity=0.3]  (0,0) -- (2,2) -- (3,1) -- cycle;
\draw[red, thick] (0,3) -- (3,0);
\draw[red, thick] (0,3) -- (1,3);
\draw[red, thick] (0,2) -- (2,2);
\draw[red, thick] (1,2) -- (1,1);
\draw[red, thick] (0,2) -- (1,1);
\draw[red,thick] (1,1) -- (2,1);
\draw[red,thick] (0,1) -- (3,1);
\draw[red,thick] (1,0) -- (1,3);
\draw[red,thick] (2,0) -- (2,2);
\draw[red,thick] (1,1) -- (2,0);
\draw[red,thick] (0,1) -- (1,0);
\draw[red,thick] (3,0) -- (3,1);
\node[text width=.2cm] at (0.48,2.72) {\tiny\textcolor{OrangeRed}{$\triangle_{00}$}};
\draw(2.9,1.7)node{{\color{ForestGreen} T}};

\draw[gray, thick] (0,0) -- (0,4);
\draw[gray, thick] (0,0) -- (4,0);
\draw[gray, thick] (0,4) -- (4,0);
\draw[gray, thick] (0,4) -- (4,0);
\foreach \Point/\PointLabel in
{
(0,4)/A,(0,3)/,(0,2)/, (0,1)/, (0,0)/,
(1,3)/, (1,2)/, (1,1)/, (1,0)/,
(2,2)/p, (2,1)/, (2,0)/, (3,0)/,
(3,1)/q, (4,0)/C}
\draw[fill=gray] \Point circle (2pt) node[above right] {$\PointLabel$};
\foreach \Point/\PointLabel in
{
(0,0)/O, 
 }
\draw[fill=black] \Point circle (2pt) node[below left] {$\PointLabel$};
\end{tikzpicture}
\caption{Barycenter division of  $k\tri_0$ and $T$. }\label{de-0}
\end{figure}
\end{center}

By adding an affine function to $g$ if necessary,  Lemma \ref{ma-0} and Corollary \ref{co-ono} implies that we only need to establish Theorem \ref{ono} for $g$ under the following additional:
\begin{assu} \label{ass}
Let $g\in \mathrm{PL}(\tri_i;k)^{W_i}, i=3,4$ satisfying:
\begin{enumerate}
\item $g(0)=\displaystyle\max_{x\in\tri_i} g(x)$;
\item $g$ vanishes on the vertices of $\tri_i$.
\end{enumerate}
\end{assu}
To achieve this,   we will establish the following two {\bf key estimates:}
\begin{itemize}
\item {\em Trapezoid for $T=\mathrm{Conv}(O,p,q)\subset\RR^2$}, the convex hull of $(O,p,q)$.
\begin{equation}\label{T-trap}
\frac{1}{\vol(T)}\int_T g\geq \frac{g(0)+g(p)+g(q)}{3}
\end{equation}
with equaliy if and only if $g$ is affine.

\item {\em  Trapezoid for standard subdivision}
\begin{eqnarray}\label{s-trap}
\int_{k\tri} g
&\geq& \frac{\vol(\tri_{00})}{3}\left(6\sum_{x\in (k\tri_i^\circ\cap \ZZ^2}  g(x)+3\sum_{x\in(k\dd \tri_i)\cap \ZZ^2}  g(x)-6\al g(0)\right)\nonumber\\
&=&\sum_{x\in (k\tri_i)^\circ\cap \ZZ^2}  g(x)+\frac{1}{2}\sum_{x\in(k\dd \tri_i)\cap \ZZ^2}  g(x)-\al g(0)\nonumber\\
&=&\sum_{x\in (k\tri_i)\cap \ZZ^2}  g(x)-\frac{1}{2}\sum_{x\in(k\dd \tri_i)\cap \ZZ^2}  g(x)-\al g(0)
\end{eqnarray}
with equaliy if and only if $g$ is affine, where  $\vol(\tri_{00})=\frac{1}{2}$ (cf. Figure  \ref{de-0})and $\al=\displaystyle\frac{6-\mathrm{ord}(\si_i)}{6}, i=3,4$.
\end{itemize}

\begin{proof}[Proof of Theorem \ref{main}]
To simplify our notation, in the rest of the proof we will use $\tri$ to denote $\tri_i,\  i=3,4$.  

Let us assume the validity of \eqref{T-trap} and \eqref{s-trap} for the moment and our goal is to
prove
\begin{equation}\label{m}
\frac{1}{\vol(k\tri)}\int_{k\tri} g\geq \frac{1}{\chi_\tri (k)} \sum_{x\in (k\tri)\cap\ZZ^n} g(x).
\end{equation}
for $g$ satisfying Assumption \ref{ass}.
By applying  the Pick formula (cf. \cite{Pic1899} and \cite{Pul1979})
$$\chi_\tri(k)=\vol(k\tri)+\frac{b}{2}+1 \text{ with } b=|(k\dd\tri)\cap\ZZ^n|$$
the left hand side of \eqref{m} can be written as
\begin{eqnarray*}
&&\left(\frac{1}{\vol(k\tri)}-\frac{1}{\chi_\tri(k)}\right)\int_{k\tri} g+ \frac{1}{\chi_\tri(k)} \int_{k\tri} g\\
(\text{ by }\eqref{s-trap}\ )&\geq &\frac{\frac{b}{2}+1}{\vol(k\tri)\cdot \chi_\tri(k)}\int_{k\tri} g+ \frac{1}{\chi_\tri(k)}\left(\sum_{x\in k\tri^\circ\cap \ZZ^n}  g(x)+\frac{1}{2}\sum_{x\in(\dd k\tri)\cap \ZZ^n}  g(x)-\al g(0)\right)\\
( \text{ by }\eqref{T-trap} \ )&\geq &
\frac{\frac{b}{2}+1}{b\cdot\vol(T)\cdot\chi_\tri(k)}\cdot\frac{\vol(T)}{3}\left(2\sum_{x\in (k\dd \tri)\cap\ZZ^n}g(x)+bg(0) \right)+\\
&&+ \frac{1}{\chi_\tri(k)}\left(\sum_{x\in k\tri^\circ\cap \ZZ^n}  g(x)+\frac{1}{2}\sum_{x\in(k\dd \tri)\cap \ZZ^n}  g(x)-\al g(0)\right)\\
&\geq &
\frac{\frac{b}{2}+1}{3b\cdot\chi_\tri(k)}\left(2\sum_{x\in (k\dd \tri)\cap\ZZ^n}g(x)+bg(0) \right)+\\
&&+ \frac{1}{\chi_\tri(k)}\left(\sum_{x\in k\tri^\circ\cap \ZZ^n}  g(x)+\frac{1}{2}\sum_{x\in(k\dd \tri)\cap \ZZ^n}  g(x)-\al g(0)\right)\\
\end{eqnarray*}
So to prove \eqref{m}, all we need is
\begin{equation}
\frac{1+b/2}{3b}\left(2\sum_{x\in (k\dd \tri)\cap\ZZ^n}g(x)+bg(0) \right)\geq\left( \frac{1}{2}\sum_{x\in(k\dd \tri)\cap \ZZ^n}  g(x)+\al g(0)\right)
\end{equation}
which is equivalent to
$$
\left(\frac{1+b/2}{3}-\al\right)g(0)\geq \left(\frac{1}{2}-\frac{1+b/2}{3b}\cdot 2\right)\sum_{x\in (k\dd \tri)\cap\ZZ^n}g(x).
$$
Using the fact $\displaystyle g(0)=\max_{x\in \tri} g(x)\geq \frac{1}{b}\sum_{x\in (k\dd \tri)\cap\ZZ^n}g(x)$,  we know that  \eqref{m} is a consequence of the following:
$$
\frac{1}{b}\geq\frac{\displaystyle\frac{1}{2}-\frac{2+b}{3b}}{\displaystyle\frac{2+b}{6b}-\al}=\frac{b-4}{b^2+(2-6\al)b}
$$
which is equivalent to $4\geq 6\al-2$. But this always hold as long as $\al=\displaystyle\frac{6-\mathrm{ord}(\si_i)}{6} \leq 1, \ i=1,2$.  And our proof of Theorem \ref{main} is completed {for $X_3$ and $X_4$}.
\end{proof}

\begin{proof}[Proof of \eqref{T-trap} and \eqref{s-trap}]
\eqref{T-trap} follows from the concavity of $g$ and trapezoidal rule. For \eqref{s-trap}, we triangulate  $\tri_0$ into the union of {\em basic triangles} $\tri_{00}$'s as illustrated  in Figure \ref{de-0} and then extend this triangulation to the whole $\tri_i$ via the Weyl group $W_i$.  Then \eqref{s-trap} follows by noticing that
\begin{enumerate}
\item each interior lattice points of $\tri_i^\circ$  that is {\em not} the point  $O$ is exactly a vertex of $6$ basic triangles of  $\tri_{00}$;
\item each boundary lattice point of $(\dd\tri)^\circ$ is is exactly a vertex of $3$ basic triangles of  $\tri_{00}$;
\item  the point $O$ is the vertex of $\mathrm{ord}(\si_i), i=1,2$ basic triangles of  $\tri_{00}$ (cf. Figure \ref{de-0}).
\end{enumerate}
And our proof of \eqref{T-trap} and \eqref{s-trap} is thus completed.
\end{proof}

\section{$X_2$.}
Recall that 
$(X_{\tri_2},L_{\tri_2})=(X_2=\PP^1\times\PP^1/(\ZZ/4\ZZ),\sO_{X_2}(-2K_{X_2}))\subset(\PP^6,\sO_{\PP^6}(1))$
 with the $\ZZ/4\ZZ=\la\xi\ra$-action generated by  $\xi\cdot ([z_1,z_2],[w_1,w_2])=([\sqrt{-1} z_1,z_2],[-\sqrt{-1} w_1,w_2])$.
Then the Weyl group $W_2=W_{\tri_2}=\ZZ/2\ZZ\times \ZZ/2\ZZ$ and 
$$
\tri_2=\conv\{(-2,0),(2,0),(0,1),(0,-1)\} \text{ (cf. Figure \ref{de2})}.
$$

It turns out this is the {\em trickest} case among all $\{X_i\}_{1\leq i\leq 4}$.
\begin{center}
\begin{figure}[h!]
\begin{tikzpicture}[x=.5cm,y=.5cm]
\draw[fill=green, opacity=0.3]  (0,0) -- (4,0) -- (2,1) -- cycle;
\draw(3.3,1.0)node{{\color{ForestGreen} $T_i$}};
\draw[red, thick] (0,2) -- (0,0);
\draw[red, thick] (0,0) -- (4,0);
\draw[red,thick] (0,2) -- (4,0);
\node[text width=1.5cm] at (2.7,1.9) {\red $\triangle_0$};
\node[text width=1cm] at (2.8,1.35) { $a_1$};
\node[text width=1cm] at (0.65,2.4) { $a_0$};
\node[text width=1cm] at (4.8,0.35) { $a_2$};

\draw[gray, thick] (0,2) -- (-4,0);
\draw[gray, thick] (-4,0) -- (0,-2);
\draw[gray, thick] (4,0) -- (0,-2);
\node[text width=1.5cm] at (-4,0) { $\triangle_2$};

\foreach \Point/\PointLabel in
{
(0,-2)/, 
(-2,-1)/, (-1,-1)/,(0,-1)/, (1,-1)/,(2,-1)/,
(-4,0)/,(-3,0)/,(-2,0)/, (-1,0)/, (0,0)/, (1,0)/,(2,0)/,(3,0)/,(4,0)/,
(-2,1)/, (-1,1)/, (0,1)/, (1,1)/,(2,1)/,
(0,2)/}
\draw[fill=gray] \Point circle (2pt) node[above right] {$\PointLabel$};
\foreach \Point/\PointLabel in
{
(0,0)/O, 
 }
\draw[fill=black] \Point circle (2pt) node[below left] {$\PointLabel$};
\end{tikzpicture}
\caption{$\triangle_0\subset k\triangle_2$ with $k=2$.}\label{de2}
\end{figure}
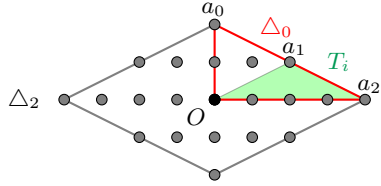
\end{center}

For this purpose, we need to extend the main estimate  \eqref{s-trap} (cf. \eqref{s-trap1}) used in the last section.
Let  $$a_i:=(2i,k-i)\in \RR^2\text{ for } 0\leq i\leq k,$$
denote the {\em integral points} of the boundary of $k\tri_0\subset k\tri_2$ (cf. Figure \ref{de2})  and  
let  
\begin{equation}\label{bT}
T_i:=\conv (0,a_i,a_{i+1}), 0\leq i\leq k-1\text{ and }b=|\dd (k\tri_2)\cap \ZZ^2| 
\end{equation}
Then $\tri_0=\displaystyle\bigcup_{i=0}^{k-1} T_i$ and  we have the following:
\begin{lemm}\label{ax}
Let $\Delta\subset \RR^2$ be a integral polytope, and $g\in \mathrm{PL}(\tri_2;k)^{W_2}$ satisfying
\begin{equation}\label{s-trap1}
\sum_{p\in (k\tri_2)\cap\ZZ^2}g(p)-\int_{k\tri_2}g(x)dV\leq \dfrac{1}{2}\sum_{p\in \dd (k\tri_2) \cap \ZZ^2}g(p)+\sum_{p\in k\tri_2}\delta_k(p)g(p)  
\end{equation} 
for a {\em fixed} function $\delta_k(p)$ satisfying  \[\sum_{p\in k\tri_2} \delta_k(p)=1\] 
with equality holding if and only if\ $g$ is constant.
Then 
\begin{equation}\label{b+2}
\dfrac{(b+2)|W_2|}{2b\cdot\vol(T)}\sum_i\int_{T_i}g\geq  \dfrac{1}{2}\sum_{p\in \dd (k\tri_2) \cap \ZZ^2}g(p)+\sum_{p\in k\tri_2}\delta_k(p)g(p),
\end{equation}
with equality holding if and only if $g$ is constant
implies 
\begin{equation}\label{mm}
\dfrac{1}{\vol(k\tri_2)}\int_{k\tri_2}g(x)dV\geq \dfrac{1}{|(k\tri_2)\cap\ZZ^2|}\sum_{p\in (k\tri_2)\cap \ZZ^2}g(p), 
\text{ }|(k\tri_2)\cap\ZZ^2|=\chi_{\tri_2}(k)
\end{equation}
 with equality holds if and only if $g$ is constant.
\end{lemm}

\begin{proof}
By \eqref{s-trap1}, we deduce that \eqref{mm} follows from  
\[\left(\dfrac{1}{\vol(k\tri_2)}-\dfrac{1}{|(k\tri_2) \cap \ZZ^2|}\right)\int_{k\tri_2}g\geq \dfrac{1}{|(k\tri_2) \cap \ZZ^2|}\left(\dfrac{1}{2}\sum_{p\in \dd (k\tri_2) \cap \ZZ^2}g(p)+\sum_{p\in( k\tri_2)\cap\ZZ^2}\delta_k(p)g(p)\right).\]
which is equivalent to
\begin{equation}\label{b+}
\left(\dfrac{|(k\tri_2) \cap \ZZ^2|}{\vol(k\tri_2)}-1\right)\int_{k\tri_2}g\geq \dfrac{1}{2}\sum_{p\in \dd (k\tri_2) \cap \ZZ^2}g(p)+\sum_{p\in (k\tri_2)\cap\ZZ^2}\delta_k(p)g(p).
\end{equation} 
By subdividing $k\tri_2$ into $b$  triangles as in Figure \ref{de2}, that is   $k\tri_2=\displaystyle\bigcup_{g\in W_2}g\cdot \tri_0$ and  $\tri_0=\displaystyle\bigcup_{i=0}^{k-1} T_i$ 
then  we have
\[\vol(k\tri_2)=b\sum_i\vol(T_i)=b\cdot\vol(T_0).\] 
Using the fact 
$b=|\dd (k\tri_2)\cap \ZZ^2|$ and pluging $g=1$ into \eqref{b+} we deduce
\[\left(\dfrac{|k\tri_2 \cap \ZZ^2|}{b\cdot\vol(T)}-1\right)b\cdot\vol(T)= {1\over 2}b+1.\] 
Hence
\[\left(\dfrac{|k\tri_2 \cap \ZZ^2|}{b\cdot\vol(T)}-1\right)=\dfrac{b+2}{2b\cdot\vol(T)},\] 
our proof is completed by plugging this into \eqref{b+}.
\end{proof}



Now to prove Theorem \ref{main}, one needs to establish the estimate \eqref{s-trap1} and \eqref{b+2} for an appropriate  $\delta_k$ in Lemma \ref{ax} (cf. \eqref{s-trap1}).

{\em Step 1. establishing \eqref{s-trap1} for an  appropriate $\de_k$.}
Using $W_2=\ZZ/2\ZZ \times \ZZ/2\ZZ$ symmetry of $k\tri_2$, it suffices to consider $\De_0$ as in Figure \ref{de2}.  Now let us do a sub-division
 \[k\tri_2=\conv\{(\pm 0,k), (\pm 2k,0)\}=\tri_{00}\cup\tri_{01} \text{ (cf. Figure \ref{de3}) }\]
 with $\tri_{00}:=\conv((0,k),(0,0),(k,0))$ and $\tri_{01}:=\conv((0,k),(k,0),(2k,0))$.
Clearly, $\tri_{00}$ is $\SL(2,\ZZ)$ equivalent to $\tri_{01}$. 
\begin{center}
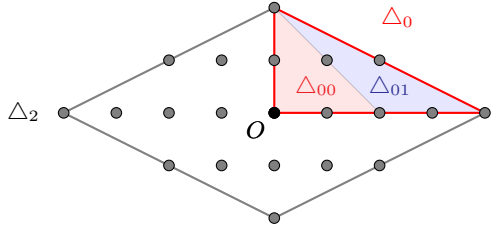
\begin{figure}[h!]
\begin{tikzpicture}[x=.70cm,y=.70cm]
\draw[fill=red, opacity=0.1]  (0,2) -- (0,0) -- (2,0) -- cycle;
\draw[fill=blue, opacity=0.1]  (0,2) -- (2,0) -- (4,0) -- cycle;
\draw(2.2,0.5)node{{\color{Blue} $\tri_{01}$}};
\draw(0.8,0.5)node{{\color{Red} $\tri_{00}$}};
\draw[red, thick] (0,2) -- (0,0);
\draw[red, thick] (0,0) -- (4,0);
\draw[red,thick] (0,2) -- (4,0);
\node[text width=1.5cm] at (3.1,1.8) {\red $\triangle_0$};

\draw[gray, thick] (0,2) -- (-4,0);
\draw[gray, thick] (-4,0) -- (0,-2);
\draw[gray, thick] (4,0) -- (0,-2);
\node[text width=1.5cm] at (-4,0) { $\triangle_2$};

\foreach \Point/\PointLabel in
{
(0,-2)/, 
(-2,-1)/, (-1,-1)/,(0,-1)/, (1,-1)/,(2,-1)/,
(-4,0)/,(-3,0)/,(-2,0)/, (-1,0)/, (0,0)/O, (1,0)/,(2,0)/,(3,0)/,(4,0)/,
(-2,1)/, (-1,1)/, (0,1)/, (1,1)/,(2,1)/,
(0,2)/}
\draw[fill=gray] \Point circle (2pt) node[below left] {$\PointLabel$};
\foreach \Point/\PointLabel in
{
(0,0)/O, 
 }
\draw[fill=black] \Point circle (2pt) node[below left] {$\PointLabel$};
\end{tikzpicture}
\caption{$\triangle_0\subset k\triangle_2$ with $k=2$.}\label{de3}
\end{figure}
\end{center}
Now let us introduce a triangulation of $\tri$ by introducing a triangulation on $\tri_0$:
\begin{itemize}
\item using the  standard triangulation of  $\tri_{00}$ (cf. Figure \ref{de-0});
\item  transporting the triangulation of $\tri_{00}$ to $\tri_{01}$ via the $\SL(2,\ZZ)$,
\end{itemize}
Applying \eqref{s-trap},  we obtain 
\[\int_{k\tri_2}g(x) \geq \sum_{p\in (k\tri_2)\cap\ZZ^2}g(p)-\dfrac{1}{2}\sum_{p\in \dd (k\tri_2)\cap\ZZ^2}g(p)+\dfrac{1}{6}g(0,\pm k)-\dfrac{1}{6}g(\pm2 k,0)-\dfrac{1}{3}(g(\pm k,0))-\dfrac{1}{3}g(0,0)\]
where 
\begin{enumerate}
\item for $g(0,\pm k)$,  we have  $\dfrac{1}{6}=\dfrac{4}{6}-\dfrac{3}{6}$ since the vertices $(0,\pm k)$ are shared by $4$ triangles instead of $3$ in the triangulation above. 
\item for $g(\pm 2k,0 )$,  we have  $-\dfrac{1}{6}=\dfrac{2}{6}-\dfrac{3}{6}$ since the vertices $(\pm 2k,0)$ are shared by $2$ triangles instead of $3$ in the triangulation above.
\item for $g(\pm k,0 )$ and $g(0,0)$,  we have  $-\dfrac{1}{3}=\dfrac{4}{6}-\dfrac{6}{6}$ since the vertices $(\pm k,0)$ are shared by $4$ triangles instead of $6$ ( as they are boundary point of $\tri_{00}$ and $\tri_{01}$ but interior points of $\tri_2$).
\end{enumerate}
Hence
\begin{eqnarray*}
\sum_{p\in \dd k\tri_2\cap\ZZ^2}g(p)-\int_{k\tri_2}g(x)
&\leq &\dfrac{1}{2}\sum_{p\in \dd (k\tri_2)\cap\ZZ^2}g(p)-\dfrac{1}{6}g(0,\pm k)+\dfrac{1}{6}g(\pm2 k,0)+\dfrac{1}{3}
(g(\pm k,0))+\dfrac{1}{3}g(0,0)\\
&\leq &\dfrac{1}{2}\sum_{p\in \dd (k\tri_2)\cap\ZZ^2}g(p)-\dfrac{1}{6}g(0,\pm k)+\dfrac{1}{6}g(\pm2 k,0)+g(0,0)
\end{eqnarray*}
since $g(0,0)=\displaystyle\max_{k\tri} g\geq g(\pm k,0)$ for $g\in \mathrm{PL}(\tri_2;k)^{W_2}$. 
Thus we established \eqref{s-trap1} with $\de_k:k\tri_2\to \RR$ defined by
\begin{align}
\de_k(p)=\left\{\begin{array}{lll} 
0 & p=(0,0)\\
-1/6  & p=(0,\pm k)\\ 
1/6  & p=(\pm 2k,0)\\
0 & \text{ otherwise}. \end{array}\right.
\end{align}


{\em Step 2, establishing \eqref{b+2}}.
That is, for all $g\in \mathrm{PL}(\tri_2;k)^{W_2}$ we need to show
	\[\dfrac{b+2}{2b\cdot \vol(T_0)}\sum_{g\in W_2}\sum_i\int_{g\cdot T_i}g\geq  \dfrac{1}{2}\sum_{p\in \dd (k\tri_2 )\cap \ZZ^2}g(p)+\sum_{p\in k\tri_2}\delta_k(p)g(p).\]
Let us first consider $ T_0= \conv((0,0),(2,k-1),(0,k))$.  
Applying \eqref{T-trap}, we have
\begin{equation}\label{t0}
\int_{T_0}g\geq \dfrac{\vol(T_0)}{3}(g(0)+g(2,k-1)+g(0,k)).
\end{equation}
 By the $W_2$-symmetry of $g$, we have $g(-2,k-1)=g(2,k-1)$, this together with the concavity of $g$ imply $ g(0,k-1)\geq g(2,k-1)$ and $g(0,k-1)\geq g(0,k)$, so 
\begin{equation}\label{k-1}
g(0,k-1)\geq \dfrac{g(2,k-1)+g(0,k)}{2}.
\end{equation}
Therefore,  
\begin{align*}
\frac{1}{\vol(T_0)}\int_{T_0} g
=&\frac{1}{\vol(T_0)}\left(\int_{T_{00}} g+\int_{T_{01}} g\right)\\
\geq & \dfrac{\vol(T_{00})}{\vol(T_{0})}\left(\dfrac{g(0,0)+g(2,k-1)+g(0,k-1)}{3}\right)
+\dfrac{\vol(T_{01})}{\vol(T_{0})}\left(\dfrac{g(0,k)+g(2,k-1)+g(0,k-1)}{3}\right)\\
\geq& \dfrac{k-1}{k}\left(\dfrac{g(0,0)+g(2,k-1)+\frac{g(2,k-1)+g(0,k)}{2}}{3}\right)
+\dfrac{1}{k}\left(\dfrac{g(0,k)+g(2,k-1)+\frac{g(2,k-1)+g(0,k)}{2}}{3}\right)\\
=&\left(\dfrac{k-1}{k}\right)\dfrac{g(0,0)}{3}+\left(\dfrac{1}{3}+\dfrac{1}{6}\right)g(2,k-1)+\left(\dfrac{1}{6}+\dfrac{1}{3k}\right)g(0,k)\\
=&\left(1-\dfrac{1}{k}\right)\dfrac{g(0,0)}{3}+\left(\dfrac{1}{3}+\dfrac{1}{6}\right)g(2,k-1)+\left(\dfrac{1}{3}-\dfrac{1}{6}+\dfrac{1}{3k}\right)g(0,k)\\
=&\dfrac{1}{3}\left(g(0,0)+g(2,k-1)+g(0,k)\right)
-\dfrac{1}{3k}g(0,0)+\dfrac{1}{6}g(2,k-1)+\left(-\dfrac{1}{6}+\dfrac{1}{3k}\right)g(0,k)
\end{align*}

\begin{center}
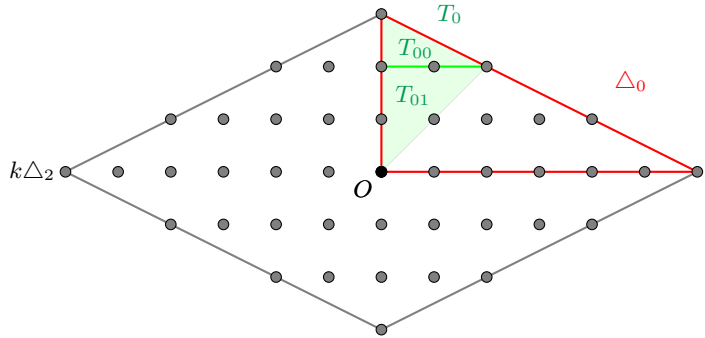
\begin{figure}[h!]
\begin{tikzpicture}[x=.70cm,y=.70cm]
\draw[fill=green, opacity=0.1]  (0,0) -- (0,3) -- (2,2) -- cycle;
\draw(1.3,3.0)node{{\color{ForestGreen} $T_0$}};
\draw(0.65,2.35)node{{\color{ForestGreen} $T_{00}$}};
\draw(0.6,1.45)node{{\color{ForestGreen} $T_{01}$}};
\draw[green, thick] (0,2) -- (2,2);
\draw[red, thick] (0,3) -- (0,0);
\draw[red, thick] (0,0) -- (6,0);
\draw[red,thick] (0,3) -- (6,0);
\node[text width=1.5cm] at (5.5,1.7) {\red $\triangle_0$};

\draw[gray, thick] (0,3) -- (-6,0);
\draw[gray, thick] (-6,0) -- (0,-3);
\draw[gray, thick] (6,0) -- (0,-3);
\node[text width=1.5cm] at (-6,0) { $k\triangle_2$};

\foreach \Point/\PointLabel in
{(0,-3)/,
(-2,-2)/, (-1,-2)/,(0,-2)/, (1,-2)/,(2,-2)/,
(-4,-1)/, (-3,-1)/,(-2,-1)/, (-1,-1)/,(0,-1)/, (1,-1)/,(2,-1)/,(3,-1)/, (4,-1)/,
(-6,0)/, (-5,0)/, (-4,0)/,(-3,0)/,(-2,0)/, (-1,0)/, (0,0)/O, (1,0)/,(2,0)/,(3,0)/,(4,0)/, (5,0)/, (6,0)/,
(-4,1)/, (-3,1)/,(-2,1)/, (-1,1)/, (0,1)/, (1,1)/,(2,1)/,(3,1)/, (4,1)/,
(-2,2)/, (-1,2)/,(0,2)/, (1,2)/,(2,2)/,
(0,3)/}
\draw[fill=gray] \Point circle (2pt) node[below left] {$\PointLabel$};
\foreach \Point/\PointLabel in
{
(0,0)/O, 
 }
\draw[fill=black] \Point circle (2pt) node[below left] {$\PointLabel$};
\end{tikzpicture}
\caption{$\triangle_0\subset k\triangle_2$ with $k=3$.}\label{de3}
\end{figure}
\end{center}


Combining the estimates with the ones for $T_i, i\ne 0$ based on \eqref{T-trap}, we obtain  
\begin{align*}
&\dfrac{(b+2)|W_2|}{2b\vol(T_0)}\sum_{i=0}^{k-1}\int_{T_i}g\\
\geq& \dfrac{b+2}{2b}\left(\dfrac{bg(0)+ 2\sum_{p\in \dd (k\Delta)}g(p)}{3}\right)\\
&+\dfrac{b+2}{2b}\left(-\dfrac{|W_2|}{3k}g(0)+\dfrac{|W_2|}{6}\cdot g(2,k-1)+\left(\dfrac{2}{3k}-\dfrac{1}{3}\right)(g(0,k)+g(0,-k))\right)\\
=&:\sum_{p\in (k\tri_2)\cap\ZZ^2} \eta(p)g(p)
\end{align*}
where $\eta: k\tri_2\cap\ZZ^2\to \RR$ is defined by the right hand side of the above inequality.

To establish \eqref{b+2}, it suffices to show	
$$
\sum_{p\in k\tri_2\cap\ZZ^2} \eta(p)g(p)\geq  \dfrac{1}{2}\sum_{p\in \dd (k\tri _2)\cap \ZZ^2}g(p)+\sum_{p\in k\tri_2}\delta_k(p)g(p),
$$
which is equivalent to 
\begin{equation}\label{DeTi}
(\eta(0)-\widetilde{\delta}_k(0))g(0)\geq \sum_{p\in ((k\tri_2) \cap \ZZ^2)\setminus \{0\}}(\widetilde{\delta}_k(p)-\eta(p))g(p)
\end{equation}
with  $\widetilde{\delta}_k(p)$ being defined by the following identity
\[ \sum_{p\in \dd k\tri_2 \cap \ZZ^2}\widetilde{\delta}_k(p)g(p)= \dfrac{1}{2}\sum_{p\in \dd (k\tri_2) \cap \ZZ^2}g(p)+\sum_{p\in k\tri_2}\delta_k(p)g(p).\] 
 As $b=4k$,  for $p\neq (0,0)$, 
$\widetilde{\delta}_k(p)-\eta(p)$ is given by  

\[(\widetilde{\delta}_k-\eta)(p)= \left\{\begin{matrix}
0 &\text{ if }& p\in k\tri_2^\circ\\
{\ }\\
\dfrac{1}{2}-\left(\dfrac{b+2}{2b}\right)\left(\dfrac{2}{3}\right)=\dfrac{1}{6}-\dfrac{1}{6k} &\text{ if }& p=(\pm 2i,\pm (k-i))\text{ for }i\neq 0,1,k\\
{\ }\\
\left(\dfrac{1}{2}+\dfrac{1}{6}\right)-\left(\dfrac{b+2}{2b}\right)\left(\dfrac{2}{3}\right)=\dfrac{1}{3}-\dfrac{1}{6k}&\text{ if }&p=(\pm 2k, 0)\\
{\ }\\
\dfrac{1}{2}-\left(\dfrac{b+2}{2b}\right)\left(\dfrac{2}{3}+\dfrac{1}{6}\right)=\dfrac{1}{12}-\dfrac{5}{24k}&\text{ if }&p=(\pm 2, \pm (k-1))\\
{\ }\\
\left(\dfrac{1}{2}-\dfrac{1}{6}\right)-\left(\dfrac{b+2}{2b}\right)\left(\dfrac{2}{3}+\left(\dfrac{2}{3k}-\dfrac{1}{3}\right)\right)
=\dfrac{1}{6}-\dfrac{5}{12k}-\dfrac{1}{6k^2}&\text{ if }&p=(0,\pm k).
\end{matrix}\right.\]
which are non-negative when $k\geq k_0$ for some $k_0$ independent of $g$. 
As a consequence,  
we have 
$$
(\eta(0)-\widetilde{\delta}_k(0))g(0)\geq \sum_{p\in (k\tri_2) \cap \ZZ^2}(\widetilde{\delta}_k(p)-\eta(p))g(0)\geq \sum_{p\in (k\tri_2) \cap \ZZ^2}(\widetilde{\delta}_k(p)-\eta(p))g(p).
$$
with equality holds if and only if $g$ is constant, and hence \eqref{b+2} is established.  The proof for the case $X_2=X_{\tri_2}$ is completed by applying Lemma \ref{ax}.

\begin{rema}
One notices that the estimate \eqref{k-1} can be improved 
\[g(0,k-1)\geq\lambda g(2,k-1)+(1-\lambda)\dfrac{k-1}{k}g(0,k),\]  
with $0\leq \lambda \leq 1$. By choosing an appropriate $\lam$, one can verify \eqref{b+2}   for $k\geq 2$.
\end{rema}

\section{$X_1$.}
Recall  $(X_{\tri_1},L_{\tri_1})=(X_1=\PP^2/(\ZZ/9\ZZ),\sO_{X_1}(-3K_{X_1}))\subset (\PP^6,\sO_{\PP^6}(1))$ with the $\ZZ/9\ZZ=\la\xi=\exp2\pi\sqrt{-1}/9\ra$-action generated by  $\xi\cdot [z_0, z_1,z_2]=[z_0,\xi,z_1,\xi^{-1},z_2]$. Then the Weyl group of $X_1$ is $W_1=\ZZ/2$ and 
$$\tri_1=\conv\{(1,2), (2,1), (-3,-3)\}\subset \RR^2 \text{(cf. Figure \ref{de3})}.$$
\begin{center}
\begin{figure}[h!]
\begin{tikzpicture}[x=1cm,y=1cm]
\draw[black, thick](-3,-3)--(1,2);
\draw[black, thick](-3,-3)--(2,1);
\draw[black, thick](2,1)--(1,2);
\draw[red, thick](-3.5,-3.5)--(2.5,2.5);
\foreach \Point/\PointLabel in
{
(-3,-3)/,(-1,-0.5)/,(-0.5,-1)/, (2,1)/, (1,2)/,(1.5,1.5)/,
(-2.5,-2.5)/, (-2,-2)/, (-1.5,-1.5)/, (-1,-1)/,
(-0.5,-0.5)/, (0,0)/, (0.5,0.5)/, (1,1)/,
(1.5,1)/, (1,0.5)/, (0.5,0)/, (0,-0.5)/,
(1,1.5)/, (0.5,1)/, (0,0.5)/, (-0.5,0)/,
 }
\draw[fill=gray] \Point circle (2pt) node[below left] {$\PointLabel$}; 
\foreach \Point/\PointLabel in
{
(0,0)/O, 
 }
\draw[fill=black] \Point circle (2pt) node[below left] {$\PointLabel$};
\end{tikzpicture}
\caption{$ k\triangle_1$ with $k=2$.}\label{de3}
\end{figure}
\end{center}

\begin{theo}$X_1$ is Chow unstable.
\end{theo} \label{x1}
To see this, first we notice  that \eqref{T-trap} implies
\begin{lemm}\label{m-int}
$\displaystyle\int_{\tri_1} xdx=(0,0)$.
\end{lemm}
By the necessity of Chow semistability \eqref{bary},  Theorem \ref{x1} follows from Lemma \ref{m-int} and the  following 	 
\begin{prop}
	  \[\frac{1}{\chi_{\tri_1}(k)}\sum_{x=(x_1,x_2)\in (k\tri_1)\cap \ZZ^2} x=\dfrac{4\cdot (-k,-k)}{9k^2+3k+2}\neq 0
	  \text{ with } \chi_{\tri_1}(k)=|(k\tri_1)\cap \ZZ^2|=\frac{9k^2+3k+2}{2}.
	  \] 
	  In particular, it violates \eqref{bary} and $X_1$ is Chow {\em unstable} for all $k\geq 1$.
	 \end{prop}	 
	 \begin{proof}
	  
By the $W_1=\ZZ/2\ZZ$-symmetry, we have
 \begin{equation}
 \frac{1}{\chi_{\tri_1}(k)}\sum_{x\in (k\tri_1)\cap \ZZ^2} x=\frac{(1,1)}{\chi_{\tri_1}(k)}\sum_{x\in (k\tri_1)\cap \ZZ^2} x_1
\end{equation}
with $x_1\in \RR$ being the first component of $x=(x_1,x_2)\in \RR^2$.

Let us define $m:=\displaystyle \frac{1}{\chi_{\tri_1}(k)}\sum_{x\in (k\tri_1)\cap \ZZ^2} x_1$. For simplicity, we will only treat the case that  $k$ is {\em even}, 
\footnote{For $k$ {\em odd}, the derivation is similar and will be left to the readers.}
then by considering the symmetry about the axis in Figure \ref{de3}, we obtain
\begin{eqnarray*}
-m&=&\frac{2}{\chi_{\tri_1}(k)} \left(\sum_{i=1}^{k/2} \frac{9(i-1)(9(i-1)+1)}{2}+\sum_{i=1}^{k/2}\left(\frac{ 5+9(i-1)}{2}+
\frac{(9(i-1)+4)(9(i-1)+5)}{2}\right)\right)\\
&&+\frac{1}{\chi_{\tri_1}(k)} \frac{\frac{9k}{2}(\frac{9k}{2}+1)}{2}-\frac{3k}{2}\\
&=&\frac{2k}{\chi_{\tri_1}(k)}.
\end{eqnarray*}

\end{proof}
	
\begin{exam}
For $k=1$,
\[ \frac{1}{\chi_{\tri_1}(1)}\sum_{x\in \tri_1\cap \ZZ^2} x=\dfrac{(-2,-2)}{7}\]
\end{exam}.
\begin{rema}\label{osn}
We remark that this example as well as the example in \cite{OSN12} have {\em not} ruled out the possibility of using the asymptotic Chow semistability to compactify the moduli space of Fano varieties contrasting to the case studied in \cite{WX14}, since for those punctured families one might have  a limit which is asymptotic Chow polystable and strict K-semistable simultaneuously. 
\end{rema}

\end{document}